\documentclass[12pt,a4paper]{article}
\usepackage{amsmath, amssymb, amsthm, eufrak, setspace, verbatim}
\usepackage[top=2.5cm, bottom=2cm, left=2cm, right=2cm]{geometry}
\usepackage[all]{xy}
\theoremstyle{definition}
\newtheorem{thm}{Theorem}[section]

\newtheorem{prop}[thm]{Proposition}
\newtheorem{lemma}[thm]{Lemma}
\newtheorem{cor}[thm]{Corollary}

\newcommand{\goth}{\mathfrak}

\newcommand{\perm}[1]{\mbox{Perm}(#1)}

\def \OL {{\goth O}_{L}}
\def \OK {{\goth O}_{K}}

\def \p  {{\goth p}}

\def \OLp {{\goth O}_{L,{\p}}}

\def \OKp {{\goth O}_{K,{\p}}}

\def \A {{\goth A}}
\def \B {{\goth B}}

\def \Q {\mathbb{Q}}
\def \Z {\mathbb{Z}}

\def \e0 {e_{0}}

\def \B {\mathfrak B}

\title{Canonical Nonclassical Hopf-Galois Module Structure of Nonabelian Galois Extensions \\ Preliminary Report}
\author{Paul J. Truman}

\begin{document}

\maketitle

\begin{abstract}
Let $ L/K $ be a finite Galois extension of local or global fields in characteristic $ 0 $ or $ p $ with nonabelian Galois group $ G $, and let $ \B $ be a $ G $-stable fractional ideal of $ L $. We show that  $ \B $ is free over its associated order in $ K[G] $ if and only if it is free over its associated order in the Hopf algebra giving the canonical nonclassical Hopf-Galois structure on the extension. 
\end{abstract}

\section{Introduction and Statement of Results}

Throughout let $ L/K $ be a finite Galois extension of fields with nonabelian Galois group $ G $. By the theorem of Greither and Pareigis
\begin{itemize}
\item the Hopf-Galois structures on $ L/K $ are in bijective correspondence with the regular subgroups of $ \perm{G} $ normalized by $ \lambda(G) $, the image of $ G $ under the left regular embedding,
\item the Hopf algebra corresponding to a regular subgroup $ N $ is $ H_{N}=L[N]^{G} $, where $ G $ acts on $ L $ as Galois automorphisms and on $ N $ by conjugation via the embedding $ \lambda $,
\item such a Hopf algebra acts on $ L $ by
\[ \left( \sum_{\eta \in N} c_{\eta} \eta \right) \cdot x = \sum_{\eta \in N} c_{\eta} \eta^{-1}(1_{G})[x] \hspace{5mm} (c_{\eta} \in L, \; x \in L). \]
\end{itemize}

Two examples of such regular subgroups are $ \lambda(G) $ itself and $ \rho(G) $, the image of $ G $ under the right regular embedding. The latter corresponds to the classical structure, with Hopf algebra $ K[G] $ and its usual action on $ L $. Since $ G $ is nonabelian we have $ \lambda(G) \neq \rho(G) $, and the subgroup $ \lambda(G) $ corresponds to a canonical nonclassical Hopf-Galois structure on $ L/K $, whose Hopf algebra we will denote by $ H_{\lambda} $. 
\\ \\
Our main result is the following:

\begin{thm} \label{thm_main}
Let $ L/K $ be a finite nonabelian Galois extension of local fields or global fields with group $ G $, and suppose that $ \B $ is a $ G $-stable fractional ideal of $ L $. Then $ \B $ is free over its associated order in $ K[G] $ if and only if it is free over its associated order in $ H_{\lambda} $. 
\end{thm} 

Note that we make no restriction on the characteristic of $ K $. Some immediate corollaries of this are:

\begin{cor}
Suppose that $ L/K $ is an extension of local fields and is at most tamely ramified. Then $ \OL $ is free over its associated order in $ H_{\lambda} $.
\end{cor}
\begin{proof}
In this case $ \OL $ is a free $ \OK[G] $-module by Noether's Theorem, so Theorem \eqref{thm_main} applies. 
\end{proof}

\begin{cor}
Suppose that $ L/K $ is an extension of global fields and is at most tamely ramified. Then $ \OL $ is locally free over its associated order in $ H_{\lambda} $.
\end{cor}
\begin{proof}
The proof of Theorem \eqref{thm_main} does not depend on the fact that $ L $ is a field, so we could replace $ L $ with its completion at some prime $ \p $ of $ \OK $ (a Galois algebra). In this case, for each prime $ \p $ of $ \OK $ we have that $ \OLp $ is a free $ \OKp[G] $-module by Noether's Theorem, so Theorem \eqref{thm_main} applies at each prime, and so $ \OL $ is locally free over its associated order in $ H_{\lambda} $. 
\end{proof}

\begin{cor}
Suppose that $ K = \Q $ and that $ L/K $ is tame and that $ [L:\Q] $ is not divisible by $ 4 $. Then $ \OL $ is free over its associated order in $ H_{\lambda} $.
\end{cor}
\begin{proof}
In this case $ \OL $ is a free $ \Z[G] $-module by Taylor's Theorem, so Theorem \eqref{thm_main} applies. 
\end{proof}

\begin{cor}
Suppose that $ L/K $ is an extension of $ p $-adic fields which is weakly ramified. Then $ \OL $ is free over its associated order in $ H_{\lambda} $.
\end{cor}
\begin{proof}
In this case $ \OL $ is free over its associated order in $ K[G] $ by a theorem of Johnston, so Theorem \eqref{thm_main} applies.
\end{proof}

\begin{cor}
Suppose that $ G $ is simple. By a result of Byott, $ L/K $ admits exactly two Hopf-Galois structures: the classical structure and the canonical nonclassical structure, and by Theorem \eqref{thm_main} $ \OL $ is either free over its associated order in both of these structures or in neither of them. 
\end{cor}

(Remember that in all of these we are assuming that $ L/K $ is Galois with nonabelian Galois group $ G $.) 

\section{Normal Basis Generators}

In this section we will prove the following theorem:

\begin{thm} \label{thm_NBG}
Let $ x \in L $. Then $ x $ is a $ K[G] $-generator of $ L $ if and only if $ x $ is an $ H_{\lambda} $-generator of $ L $. 
\end{thm}

To do this, for this section only we place ourselves in a slightly more general situation, and adopt the notation used in the proof of the theorem of Greither and Pareigis in Childs: Taming Wild Extensions, Chapter 2. 
\begin{itemize}
\item Let $ N $ be any regular subgroup of $ \perm{G} $ that is stable under the action of $ G $ by conjugation via the left regular embedding $ \lambda $. 
\item Let $ GL = \mbox{Map}(G,L) $, and let $ \{ u_{g} \mid g \in G \} $ be an $ L $-basis of mutually orthogonal idempotents. That is:
\[ u_{g}(\sigma) = \delta_{g,\sigma} \mbox{ for all } g,\sigma \in G. \]
\item The group $ N $ acts on $ GL $ by permuting the subscripts of the idempotents $ u_{g} $:
\[ \eta \cdot u_{g} = u_{\eta(g)} \mbox{ for any } \eta \in N \mbox{ and } g \in G. \]
By extending this action $ L $-linearly, we can view $ GL $ as an $ L[N] $-module.
\item As described above, $ G $ acts on $ L[N] $ by acting on $ L $ as Galois automorphisms and on $ N $ by conjugation via $ \lambda $. The group $ G $ also acts on $ GL $ by acting on $ L $ as Galois automorphisms and on the idempotents $ u_{g} $ by left translation of the subscripts.  
\item We have that $ GL $ is an $ L[N] $-Galois extension of $ L $ and, by Galois descent, we obtain that $ (GL)^{G} $ is an $ L[N]^{G} $-Galois extension of $ K $. Note also that $ L \otimes_{K} L[N]^{G} = L[N] $ and $ L \otimes_{K} (GL)^{G} = GL $. 
\item Finally, we identify $ (GL)^{G} $ with $ L $ via the isomorphism $ L \xrightarrow{\sim}{} (GL)^{G} $ defined by
\[ x \mapsto f_{x} =  \sum_{g \in G} g(x) u_{g} \mbox{ for all } x \in L. \]
The action of $ L[N]^{G} $ on $ L $ (as given in the statement of the theorem of Greither and Pareigis) is defined via the inverse of this isomorphism. 
\end{itemize}

With all this notation to hand, we establish two lemmas concerning normal basis generators and then prove Theorem \eqref{thm_NBG}. 

\begin{lemma} \label{lem_fixed_generators}
An element $ f_{x} \in (GL)^{G} $ is an $ L[N]^{G} $-generator of $ (GL)^{G} $ if and only if it is an $ L[N] $-generator of $ GL $. 
\end{lemma}
\begin{proof}
Let $ \{ h_{1}, \ldots ,h_{n} \} $ be a $ K $-basis of $ L[N]^{G} $, and note that this is also an $ L $-basis of $ L[N] $. Suppose first that $ f_{x} $ is an $ L[N]^{G} $ generator of $ (GL)^{G} $.  Then the $ K $-span of the elements $ h_{1} \cdot f_{x}, \ldots ,h_{n} \cdot f_{x} $ is $ (GL)^{G} $, so the $ L $-span of these elements is $ L \otimes_{K} (GL)^{G} = GL $. By considering dimensions we see that they must form an $ L $-basis of $ GL $. Conversely, suppose that $ f_{x} $ is an $ L[N] $-generator of $ GL $. Then the elements $ h_{1} \cdot f_{x}, \ldots ,h_{n} \cdot f_{x} $ are linearly independent over $ L $, so they are linearly independent over $ K $, and since $ (GL)^{G} $ is an $ L[N]^{G} $-module they all lie in $ (GL)^{G} $. Considering dimensions again, we conclude that they must form a $ K $-basis of $ (GL)^{G} $. 
\end{proof}

\begin{lemma} \label{lem_transition_matrix}
For $ x \in L $, the element $ f_{x} $ is an $ L[N] $-generator of $ GL $ if and only if the matrix 
\[ T_{N}(x) = ( \eta(g)[x] )_{\eta \in N, \; g \in G} \]
is nonsingular. 
\end{lemma}
\begin{proof}
The set $ \{ u_{g} \mid g \in G \} $ is an $ L $-basis of $ GL $. For $ x \in L $ and $ \eta \in N $, we have
\begin{eqnarray*}
n \cdot f_{x} & = & \eta \cdot \left( \sum_{g \in G} g(x) u_{g} \right) \\
& = & \sum_{g \in G} g(x) u_{\eta(g)} \\
& = & \sum_{g \in G} \eta^{-1}(g)[x] u_{g},
\end{eqnarray*}
so the transition matrix from the set $ \{ u_{g} \mid g \in G \} $ to the set $ \{ \eta \cdot f_{x} \mid \eta \in N \} $ is the matrix $ T_{N}(x) $ above, and so $ f_{x} $ is an $ L[N] $-generator of $ GL $ if and only if this matrix is nonsingular. 
\end{proof}

{\it \noindent Proof of Theorem \eqref{thm_NBG}.}
By the theorem of Greither and Pareigis the classical Hopf-Galois structure on $ L/K $ corresponds to the regular subgroup $ \rho(G) $ of $ \perm{G} $ and the canonical nonclassical Hopf-Galois structure corresponds to the regular subgroup $ \lambda(G) $. By Lemma \eqref{lem_fixed_generators}, it is sufficient to show that for a fixed $ x \in L $, the element $ f_{x} $ is an $ L[\lambda(G)] $-generator of $ GL $ if and only if it is an $ L[\rho(G)] $-generator of $ GL $. But for any $ x \in L $, the matrix $ T_{\lambda(G)}(x) $ is row equivalent to the transpose of the matrix $ T_{\rho(G)}(x) $, so the result follows by Lemma \eqref{lem_transition_matrix}. \qed  

\section{Three Lemmas}

Henceforth, we will reserve the symbol $ \cdot $ for the action of an element $ h \in H_{\lambda} $ on an element $ x \in L $, viz. $ h \cdot x $, and use brackets for Galois actions and the action of an element $ z \in K[G] $ on an element $ x \in L $, viz. $ z(x) $. In this section we prove three lemmas which we will need in the proof of theorem \eqref{thm_main}. The first of these must be well known but we include it for completeness:

\begin{lemma} \label{lem_dual_basis}
Let $ x $ be a $ K[G] $-generator of $ L $ and let $ \{ \widehat{\sigma(x)} \mid \sigma \in G \} $ be the dual basis to $ \{ \sigma(x) \mid \sigma \in G \} $ with respect to the trace form on $ L/K $. Then, for each $ \sigma \in G $, we have $ \widehat{\sigma(x)} = \sigma(\widehat{x}) $. 
\end{lemma} 
\begin{proof}
For $ \sigma, \tau \in G $ we have:
\begin{eqnarray*}
\mbox{Tr}_{L/K}(\sigma(\widehat{x})\tau(x)) & = & \sum_{g \in G} g(\sigma(\widehat{x})\tau(x)) \\
& = & \sum_{g \in G} g\sigma(\widehat{x})g\tau(x) \\
& = & \sum_{g \in G} (g\sigma)(\widehat{x})(g\sigma) \sigma^{-1}\tau(x) \\
& = & \sum_{g \in G} (g\sigma)(\widehat{x}\sigma^{-1}\tau(x)) \\
& = & \sum_{g \in G} g(\widehat{x}\sigma^{-1}\tau(x)) \\
& = & \mbox{Tr}_{L/K}(\widehat{x}\sigma^{-1}\tau(x)) \\
& = & \delta_{1,\sigma^{-1}\tau} \\
& = & \delta_{\sigma,\tau}. 
\end{eqnarray*}
\end{proof}

We might view the second lemma as an ``inside out" version of the first:

\begin{lemma} \label{lem_inside_out_trace}
Retain the notation of Lemma \eqref{lem_dual_basis}. Then for any $ \sigma, \tau \in G $ we have
\[ \sum_{g \in G} \sigma g (\widehat{x}) \tau g (x) = \delta_{\sigma,\tau}. \]
\end{lemma}
\begin{proof}
Enumerate the elements of $ G $ as $ g_{1}, \ldots ,g_{n} $, let $ X $ be the matrix with $ (i,j) $ entry $ ( g_{i}g_{j}(x) ) $, and let $ \widehat{X} $ be the matrix with $ (i,j) $ entry $ ( g_{i}g_{j}(\widehat{x})) $. Then using Lemma \eqref{lem_dual_basis} we have 
\[ \sum_{k=1}^{n} g_{k} g_{i}(x) g_{k}g_{j}(\widehat{x})  = \delta_{i,j}, \]
so $ X^{T} \widehat{X} = I $. But this implies that $ \widehat{X} X^{T} = I $, and the $ (i,j) $ entry of this product is given by
\[ \sum_{k=1}^{n} g_{i}g_{k}(\widehat{x}) g_{j} g_{k}(x), \]
so this must also equal $ \delta_{i,j} $. 
\end{proof}

The third lemma tells us how the action of $ H_{\lambda} $ on $ L $ interacts with the action of $ K[G] $:

\begin{lemma} \label{lem_interchange_action}
Let $ t \in L $, $ z \in K[G] $ and $ h \in H_{\lambda} $. Then
\[ h \cdot z(t) = z(h \cdot t). \]
\end{lemma}
\begin{proof}
The map $ T: L[\lambda(G)] \rightarrow L[\lambda(G)]^{G} = H_{\lambda} $ defined by
\[ z \mapsto \sum_{g \in G} \,^{g}\!z \]
is $ K $-linear and surjective, so it is sufficient to consider the case where $ h = T(y \lambda(\tau)) $ for some $ y \in L $ and $ \tau \in G $ and $ z=\sigma \in G $. In this case we have:
\begin{eqnarray*}
\sigma (T(y \lambda(\tau)) \cdot t) & = & \sigma \left( \sum_{g \in G} g(y) \,^{g}\!\lambda(\tau)  \cdot t \right)\\
& = & \sigma \left( \sum_{g \in G} g(y) \lambda(g\tau g^{-1})  \cdot t \right)\\
& = & \sigma \left( \sum_{g \in G} g(y) g\tau^{-1} g^{-1} (t) \right) \\
&& \mbox{(since } (\lambda(g\tau g^{-1}))^{-1} (1_{G}) = g\tau^{-1} g^{-1} \mbox{)}\\
& = & \sum_{g \in G} \sigma g(y) \sigma g\tau^{-1} g^{-1} (t) \\
& = & \sum_{g \in G} \sigma g(y) \sigma g\tau^{-1} g^{-1} \sigma^{-1} \sigma  (t) \\
& = & \sum_{g \in G} \sigma g(y) \,^{ (\sigma g)}\!\lambda(\tau) \cdot \sigma(t) \\
& = & \sum_{g \in G} g(y) \,^{g}\!\lambda(\tau) \cdot \sigma(t) \\
& = & T(y \lambda(\tau)) \cdot \sigma(t),
\end{eqnarray*}
as claimed. 
\end{proof}

\newpage

\section{Proof of the Main Theorem}

Let $ \B $ be a $ G $-stable fractional ideal of $ L $. Write $ \A_{K[G]} $ for the associated order of $ \B $ in $ K[G] $ and $ \A_{\lambda} $ for the associated order of $ \B $ in $ H_{\lambda} $. We shall split the ``if" and ``only if" implications of Theorem \eqref{thm_main} into two separate propositions. 

\begin{prop} \label{prop_classical_implies_nonclassical}
Suppose that $ x \in \B $ generates $ \B $ as an $ \A_{K[G]} $-module. Then $ x $ generates $ \B $ as a $ \A_{\lambda} $-module. 
\end{prop}
\begin{proof}
Since $ x $ generates $ \B $ as an $ \A_{K[G]} $-module, it generates $ L $ as a $ K[G] $-module, so $ \{ \sigma(x) \mid \sigma \in G \} $ is a $ K $-basis of $ L $. By Lemma \eqref{lem_dual_basis}, there exists $ \widehat{x} \in L $ such that $ \{ \sigma(\widehat{x}) \mid \sigma \in G \} $ is the dual basis to $ \{ \sigma(x) \mid \sigma \in G \} $. That is:
\[ \sum_{g \in G} g\sigma(\widehat{x})g\tau(x) = \delta_{\sigma,\tau} \mbox{ for all } \sigma,\tau \in G. \]
Also, there exist $ a_{1}, \ldots ,a_{n} \in \A_{K[G]} $ such that $ \{ a_{1}(x), \ldots ,a_{n}(x) \} $ is an $ \OK $-basis of $ \B $. For each $ i=1, \ldots, n $, write $ x_{i}=a_{i}(x) $ and define an element $ h_{i} \in L[\lambda(G)] $ by
\[ h_{i} = \sum_{g \in G} \left( \sum_{\rho \in G} \rho(x_{i}) g^{-1} \rho(\widehat{x}) \right) \lambda(g). \]
For each $ i=1, \ldots ,n $ we make three claims about the element $ h_{i} $:
\begin{enumerate}
\item $ h_{i} \in L[\lambda(G)]^{G} = H_{\lambda} $ (so it makes sense to let $ h_{i} $ act on an element of $ L $ using the formula given in the theorem of Greither and Pareigis). 
\item $ h_{i} \cdot x = x_{i} $ (so $ x $ is an $ H_{\lambda} $-generator of $ L $, but we knew this anyway from Theorem \eqref{thm_NBG}).
\item $ h_{i} \in \A_{\lambda} $.  
\end{enumerate}
If we can establish these three claims, then it will follow that $ \{ h_{i} \mid i=1, \ldots ,n \} $ is an $ \OK $-basis of $ \A_{\lambda} $ and that $ \B $ is a free $ \A_{\lambda} $-module. 
\\ \\
To prove (1), let $ \tau \in G $. Then 
\begin{eqnarray*}
\,^{\tau}\! h_{i} & = & \,^{\tau}\! \left( \sum_{g \in G} \left( \sum_{\rho \in G} \rho(x_{i}) g^{-1} \rho(\widehat{x}) \right) \lambda(g) \right) \\
& = & \sum_{g \in G} \tau \left( \sum_{\rho \in G} \rho(x_{i}) g^{-1} \rho(\widehat{x}) \right) \,^{\tau}\! \lambda(g) \\
& = & \sum_{g \in G} \left( \sum_{\rho \in G} \tau\rho(x_{i}) \tau g^{-1} \rho(\widehat{x}) \right) \lambda(\tau g \tau^{-1}) \\
& = & \sum_{g^{\prime} \in G} \left( \sum_{\rho \in G} \tau\rho(x_{i}) (g^{\prime})^{-1} \tau \rho(\widehat{x}) \right) \lambda(g^{\prime}) \\
&& \mbox{(writing } g^{\prime} = \tau g \tau^{-1}, \mbox{ so that } \tau g^{-1} = (g^{\prime})^{-1} \tau \mbox{)} \\
& = & \sum_{g \in G} \left( \sum_{\rho \in G} \tau\rho(x_{i}) g^{-1} \tau \rho(\widehat{x}) \right) \lambda(g) \\
&& \mbox{(replacing } g^{\prime} \mbox{ by } g \mbox{)} \\
& = & \sum_{g \in G} \left( \sum_{\rho \in G} \rho(x_{i}) g^{-1} \rho(\widehat{x}) \right) \lambda(g) \\
&& \mbox{(replacing } \tau \rho \mbox{ by } \rho \mbox{)} \\
& = & h_{i},
\end{eqnarray*} 
so $ h_{i} \in L[\lambda(G)]^{G} = H_{\lambda} $. 
\\ \\
Now we know that it makes sense to let $ h_{i} $ act on $ x $, and so we can prove (2):
\begin{eqnarray*}
h_{i} \cdot x & = & \left( \sum_{g \in G} \left( \sum_{\rho \in G} \rho(x_{i}) g^{-1} \rho(\widehat{x}) \right) \lambda(g) \right) \cdot x \\
& = & \sum_{g \in G} \left( \sum_{\rho \in G} \rho(x_{i}) g^{-1} \rho(\widehat{x}) \right) g^{-1}(x) \\
& = & \sum_{\rho \in G} \rho(x_{i}) \left( \sum_{g \in G}  g^{-1} \rho(\widehat{x})  g^{-1}(x) \right) \\
& = & \sum_{\rho \in G} \rho(x_{i}) \mbox{Tr}_{L/K} ( \rho(\widehat{x}) x ) \\
& = & \sum_{\rho \in G} \rho(x_{i}) \delta_{\rho,1} \\
& = & x_{i}. 
\end{eqnarray*}
Finally, we prove (3). It is sufficient to prove that $ h_{i} \cdot x_{j} \in \B $ for each $ j=1, \ldots ,n $.  Recall that $ x_{j} = a_{j}(x) $ for some $ a_{j} \in \A_{K[G]} $. Using Lemma \eqref{lem_interchange_action} we have:
\begin{eqnarray*}
h_{i} \cdot x_{j} & = & h_{i} \cdot a_{j}(x) \\
& = & a_{j} ( h_{i} \cdot x ) \\
& = & a_{j}(x_{i}),
\end{eqnarray*}
and this lies in $ \B $ since $ x_{i} \in \B $ and $ a_{j} \in \A_{K[G]} $. 
\\ \\
We have verified all three claims, and so the proof is complete. 
\end{proof}

The next proposition is the converse of the previous one: 

\begin{prop} \label{prop_nonclassical_implies_classical}
Suppose that $ x \in \B $ generates $ \B $ as an $ \A_{\lambda} $-module. Then $ x $ generates $ \B $ as an $ \A_{K[G]} $-module.
\end{prop}
\begin{proof}
Since $ x $ generates $ \B $ as an $ \A_{\lambda} $-module, it generates $ L $ as an $ H_{\lambda} $-module, and so by Theorem \eqref{thm_NBG} it generates $ L $ as a $ K[G] $-module. Therefore $ \{ \sigma(x) \mid \sigma \in G \} $ is a $ K $-basis of $ L $ and by Lemma \eqref{lem_dual_basis} there exists $ \widehat{x} \in L $ such that  $ \{ \sigma(\widehat{x}) \mid \sigma \in G \} $ is the dual basis to $ \{ \sigma(x) \mid \sigma \in G \} $. Mirroring the proof of Proposition \eqref{prop_classical_implies_nonclassical}, there exist $ h_{1}, \ldots ,h_{n} \in \A_{\lambda} $ such that $ \{ h_{1}\cdot x, \ldots ,h_{n} \cdot x \} $ is an $ \OK $-basis of $ \B $. For each $ i=1, \ldots, n $, write $ x_{i}=h_{i} \cdot x $ and define an element $ a_{i} \in K[G] $ by
\[ a_{i} = \sum_{g \in G} \mbox{Tr}_{L/K}(x_{i}g(\widehat{x})) g. \]
In this case it is clear that $ a_{i} \in K[G] $, so it makes sense to let $ a_{i} $ act on an element of $ L $, and we only make two claims about $ h_{i} $:
\begin{enumerate}
\item $ a_{i}(x) = x_{i} $.
\item $ a_{i} \in \A_{K[G]} $.  
\end{enumerate}
As in the proof of \eqref{prop_classical_implies_nonclassical}, if we can establish these claims then it will follow that $ \{ a_{i} \mid i=1, \ldots ,n \} $ is an $ \OK $-basis of $ \A_{K[G]} $ and that $ \B $ is a free $ \A_{K[G]} $-module. 
\\ \\
First we prove (1). We have:
\begin{eqnarray*}
a_{i}(x) & = & \sum_{g \in G} \mbox{Tr}_{L/K}(x_{i}g(\widehat{x})) g(x) \\
& = & \sum_{g \in G} \sum_{\sigma \in G} \sigma(x_{i}) \sigma g (\widehat{x}) g(x) \\
& = & \sum_{\sigma \in G} \sigma(x_{i}) \sum_{g \in G}  \sigma g (\widehat{x}) g(x) \\
& = & \sum_{\sigma \in G} \sigma(x_{i}) \delta_{\sigma,1} \mbox{ (using Lemma \eqref{lem_inside_out_trace}} ) \\
& = & x_{i}.
\end{eqnarray*}
To prove (2), it is sufficient to prove that $ a_{i}(x_{j}) \in \B $ for each $ j=1, \ldots ,n $.  Recall that $ x_{j} = h_{j} \cdot x $ for some $ h_{j} \in \A_{\lambda} $. Using Lemma \eqref{lem_interchange_action} we have:
\begin{eqnarray*}
a_{i}(x_{j}) & = & a_{i}( h_{j} \cdot x) \\
& = & h_{j} \cdot ( a_{i}(x) ) \\
& = & h_{j} \cdot x_{i},
\end{eqnarray*}
and this lies in $ \B $ since $ x_{i} \in \B $ and $ h_{j} \in \A_{\lambda} $.
\\ \\
We have verified both the claims, and so the proof is complete. 
\end{proof}

By combining Propositions \eqref{prop_classical_implies_nonclassical} and \eqref{prop_nonclassical_implies_classical}, we obtain Theorem \eqref{thm_main}

\section{Further Questions and Possible Generalizations}
Does assuming that one of $ \A_{K[G]} $ or $ \A_{\lambda} $ is a Hopf order imply that other is too? This might be particularly interesting in the case that $ L/K $ is tame and $ \B = \OL $, since then $ \A_{K[G]} = \OK[G] $, which is certainly a Hopf order. In a similar direction, if $ L/K $ is a Galois extension of $ p $-adic fields and $ p \nmid [L:K] $ then $ \OK[G] $ is a maximal order in $ K[G] $: does this imply that the associated order of $ \OL $ in $ H_{\lambda} $ is also maximal? One way to do this would be to show that it is self dual with respect to some symmetric associative bilinear form, and showing that it is Hopf would certainly suffice for this. 
\\ \\
I think that some of the nice properties of $ H_{\lambda} $ such as those expressed in Theorem \eqref{thm_NBG} and Lemma \eqref{lem_interchange_action} might boil down to the fact that $ \lambda(G) $ commutes with $ \rho(G) $ inside $ \perm{G} $. Perhaps a similar approach would work for other regular subgroups $ N $ of $ \perm{G} $ that satisfy this condition? In the local case, perhaps it would be sufficient to have some of these nice properties hold modulo $ \p_{K} $ and then argue using Nakayama's lemma?

\end{document}